\documentclass[11pt]{amsart}
\usepackage[T1]{fontenc}
\usepackage{lmodern}
\usepackage{amsmath,amssymb,amsthm,mathtools}
\usepackage[margin=30mm]{geometry}
\usepackage{microtype,booktabs}
\usepackage[colorlinks=true,linkcolor=blue,citecolor=blue,urlcolor=blue]{hyperref}
\newtheorem{theorem}{Theorem}[section]
\newtheorem{proposition}[theorem]{Proposition}
\newtheorem{lemma}[theorem]{Lemma}
\newtheorem{corollary}[theorem]{Corollary}
\theoremstyle{definition}

\theoremstyle{remark}
\newtheorem{remark}[theorem]{Remark}
\newtheorem{example}[theorem]{Example}
\numberwithin{equation}{section}
\newcommand{\C}{\mathbb C}
\newcommand{\Z}{\mathbb Z}
\newcommand{\Gm}{\mathbb G_m}
\newcommand{\cmin}{\mathbf c}

\newcommand{\Wf}{W_{0}}
\newcommand{\Wa}{W_{\mathrm a}}
\newcommand{\Fe}{F_e}
\newcommand{\Pic}{\operatorname{Pic}}
\newcommand{\Extc}{\operatorname{Ext}^{1}_{\mathrm{cent}}}
\newcommand{\HH}{\mathcal H}
\newcommand{\A}{\mathcal A}
\newcommand{\Be}{\mathcal B_e}
\newcommand{\g}{\mathfrak g}
\newcommand{\bfr}{\mathfrak b}
\newcommand{\OO}{\mathcal O}

\newcommand{\Mat}{\operatorname{Mat}}
\newcommand{\Irr}{\operatorname{Irr}}
\newcommand{\Ad}{\operatorname{Ad}}

\newcommand{\rank}{\operatorname{rank}}
\newcommand{\SL}{\operatorname{SL}}
\newcommand{\GL}{\operatorname{GL}}
\newcommand{\Sp}{\operatorname{Sp}}

\newcommand{\der}{\mathrm{der}}

\newcommand{\id}{\operatorname{id}}
\newcommand{\amin}{\mathbf a}
\title[The second lowest two-sided cell]{Number of left cells in the second lowest two-sided cell}
\author{Xun Xie}
\author{Qifan Zhou}
\date{October 23, 2026}
\begin{document}
\begin{abstract}
We prove Shi's conjecture that the number of left cells in the second lowest two-sided cell of the affine Weyl group attached to a simply connected simple complex group $G$ is exactly $|W_0|/2$, where $W_0$ is the corresponding finite Weyl group. The same formula holds for the corresponding extended affine Weyl group. The key step is to show that the reductive centralizer in $G$ of a minimal nilpotent element has identity component with simply connected derived subgroup and admits no nontrivial algebraic central extensions by $\mathbb G_m$. This gives an isomorphism of based rings between the asymptotic ring of the extended cell and a matrix ring over the representation ring of the centralizer, with matrix size equal to the Euler characteristic of the Springer fiber, namely $|W_0|/2$.
\end{abstract}
\maketitle
\section{Introduction}\label{sec:introduction}

Let $G$ be a simply connected simple complex algebraic group, with Lie algebra $\g$, a fixed maximal torus $T\subset G$, and root system $\Phi$. Let $\Wf=N_G(T)/T$ be its finite Weyl group. Put $P=X^*(T)$ for the character lattice and $Q=\Z\Phi$ for the root lattice. Then $\Wa=\Wf\ltimes Q$ is the affine Weyl group and $W=\Wf\ltimes P$ is the extended affine Weyl group attached to $G$. Since we use the character lattice of $G$, the affine Coxeter graph is of the dual type. For example, when $G=\Sp_{2n}$ with $n\geq3$, $\Wa$ has Coxeter graph of type $\widetilde B_n$.

Cells are defined for any Coxeter group using the Kazhdan--Lusztig basis \cite{KL79}; the definition extends to $W$ by adjoining the length-zero elements. Lusztig established a one-to-one correspondence between two-sided cells of $W$ and nilpotent orbits in $\g$ \cite[Theorem~4.8]{Lus89}. Bezrukavnikov proved that this correspondence is compatible with the two-sided cell order and the closure order on nilpotent orbits \cite[Theorem~4(b)]{Bez09}. The cell $\mathbf c_0$ corresponding to the zero orbit is the lowest two-sided cell, studied by Shi in \cite{Shi87}. Since $\g$ is simple, it has a unique minimal nonzero nilpotent orbit \cite[Theorem~4.3.3]{CM93}. The corresponding two-sided cell is the second lowest two-sided cell.

This paper concerns the number of left cells in the second lowest two-sided cell. Shi conjectured that this number is $|\Wf|/2$ for $\Wa$ \cite[Conjecture~0.2]{Shi11}. He proved the upper bound $|\Wf|/2$ in \cite[Theorem~5.10]{Shi11}, and subsequently established the conjecture in type $\widetilde B_n$ \cite[Theorem~4.4]{Shi23}. The second lowest two-sided cells of $W$ and $\Wa$ contain the same number of left cells. We therefore work mainly with $W$. We prove Shi's conjecture uniformly.

\begin{theorem}\label{thm:count}
The second lowest two-sided cell of the extended affine Weyl group $W$ contains exactly $|\Wf|/2$ left cells. 
\end{theorem}

Our proof uses the asymptotic ring introduced by Lusztig in \cite{Lus87}, also called the asymptotic Hecke algebra, based ring or $J$-ring. For the lowest two-sided cell $\mathbf c_0$ of $W$, Xi proved an isomorphism of based rings \cite[Theorem~1.10]{Xi90}
\[
                  J_{\mathbf c_0}\simeq\Mat_{|\Wf|}(R(G)),
\]
where $R(G)$ is the representation ring of $G$. The matrix size is exactly the number of left cells in $\mathbf c_0$, namely $|\Wf|$ \cite{Shi87}.

For a nilpotent element $e\in\g$, choose an $\mathfrak{sl}_2$-triple $(e,h,f)$ and put
\(
                         \Fe=Z_G(e,h,f).
\)
This is a maximal reductive subgroup of $Z_G(e)$, which we call the reductive centralizer of $e$. Write $R(\Fe)$ for the representation ring of finite-dimensional rational representations of $\Fe$, and let
\(
 \Be=\{\bfr\subset\g:\bfr\text{ is a Borel subalgebra and }e\in\bfr\}
\)
be the Springer fiber.

Let $\mathbf c_e$ be the two-sided cell corresponding to $e$. Lusztig conjectured that its asymptotic ring $J_{\mathbf c_e}$ is isomorphic, as a based ring, to $K_{\Fe}(Y_e\times Y_e)$ for some finite $\Fe$-set $Y_e$ \cite[Conjecture~10.5]{Lus89}. Here $K_{\Fe}(Y_e\times Y_e)$ is the Grothendieck group of $\Fe$-equivariant vector bundles on $Y_e\times Y_e$, equipped with the convolution product.

Bezrukavnikov and Ostrik proved a weaker form of this conjecture by allowing central extensions of the point stabilizers \cite{BO04}. Later, Bezrukavnikov, Dawydiak, and Dobrovolska \cite{BDD23}, and Qiu and Xi \cite{QX23}, showed that these central extensions cannot be omitted in general. Their example is the two-sided cell of type $\widetilde B_3$ corresponding to the nilpotent orbit with Jordan type $(2,2,2)$ in $\Sp_6$. A second aim of this paper is to determine whether such extensions are needed for the second lowest two-sided cell.

We answer this question through the structure of the reductive centralizer of a minimal nilpotent element. For an algebraic group $K$, let $\Extc(K,\Gm)$ denote the group of equivalence classes of algebraic central extensions of $K$ by $\Gm$. We prove the following structural result.

\begin{theorem}\label{thm:centralizer-structure}
Let $e$ belong to the minimal nonzero nilpotent orbit of $G$. Then
\[
       \Fe=\Fe^\circ Z(G),\qquad
       [\Fe^\circ,\Fe^\circ]\text{ is simply connected},
\]
and
\[
                    \Extc(\Fe,\Gm)=0.
\]
\end{theorem}

We use the  criterion that every algebraic central extension of a connected reductive group $K$ by $\Gm$ splits if its derived group $K^{\der}$ is simply connected; see \cite{Ive76,Ros21,Pop23}. Section~\ref{sec:extensions} gives a direct proof of this criterion and treats the disconnected case needed here.

As explained in Section~\ref{sec:geometry}, the finite set $Y_e$ may be chosen so that $Z(G)$ acts trivially. Since $\Fe^\circ$ also acts trivially on a finite set, Theorem~\ref{thm:centralizer-structure} implies that, for $e$ in the minimal orbit, every point of $Y_e$ is fixed by $\Fe$. Each stabilizer is therefore $\Fe$, and the associated central extensions split. The weaker form of Lusztig's conjecture proved by Bezrukavnikov and Ostrik then identifies the asymptotic ring with a matrix ring over $R(\Fe)$. Computing its size from the Springer fiber yields the following result.

\begin{theorem}\label{thm:matrix}
Let $e$ belong to the minimal nonzero nilpotent orbit of $G$, and let $\cmin$ be the second lowest two-sided cell of $W$. Then there is an isomorphism of based rings
\[
             J_{\cmin}\simeq \Mat_N(R(\Fe)),
             \qquad N=\chi(\Be)=\frac{|\Wf|}{2}.
\]
The basis $\{t_w:w\in\cmin\}$ on the left corresponds to the basis $\{E_{ij}\otimes[V]\}$ on the right, where $E_{ij}$, $1\leq i,j\leq N$, are the matrix units and $V$ runs over the irreducible rational representations of $\Fe$.
\end{theorem}

The based rings in affine type $\widetilde A$ were studied by Xi, who proved Lusztig's conjecture in that setting \cite{Xi02}. For the second lowest cell in type $\widetilde B_3$, Qiu constructed an explicit based-ring isomorphism \cite[Theorem~3.4]{Qiu25}.

More recently, Parkinson determined the asymptotic ring of the second lowest cell in type $\widetilde B_n$ for all $n\geq2$ \cite[Theorem~1.6]{Par26}. His work gives an explicit cell factorisation and determines the image of every basis element $t_w$ as a matrix unit tensored with an irreducible representation of $\Sp_{2n-2}\times\mu_2$. Theorem~\ref{thm:matrix} establishes the existence of a based-ring isomorphism uniformly in all types by a geometric argument. It does not provide an explicit rule assigning to each $w\in\cmin$ the corresponding indices $i,j$ and irreducible representation $V$.

For the minimal orbit, the Springer fiber is a union of Schubert cells indexed by precisely half the elements of $\Wf$ \cite[Sections~3.1--3.2]{AC16}. Together with the construction of $Y_e$, this gives $|Y_e|=\chi(\Be)=|\Wf|/2$, as explained in Sections~\ref{sec:geometry} and~\ref{sec:springer}.

The based-ring isomorphism also gives the number of left cells. The identity of $J_{\cmin}$ is a sum of basis elements, one for each left cell, whereas the identity of $\Mat_N(R(\Fe))$ has exactly $N$ such summands. Thus Theorem~\ref{thm:matrix} implies the assertion of Theorem~\ref{thm:count} for $W$. The correspondence recalled in Section~\ref{sec:cells} then gives the assertion for $\Wa$. These arguments are completed in Section~\ref{sec:proof}.

\section{Cells and based rings}\label{sec:cells}

Let $S$ be the set of simple reflections of $\Wa$, and let $\ell$ be its length function. Put $\A=\Z[v,v^{-1}]$. The Hecke algebra $\HH_{\mathrm a}$ of $(\Wa,S)$ has standard basis $\{T_w:w\in\Wa\}$, with
\[
 T_s^2=1+(v-v^{-1})T_s\quad(s\in S),\qquad
 T_xT_y=T_{xy}\quad\text{if }\ell(xy)=\ell(x)+\ell(y).
\]
The extended group has a decomposition $W=\Omega\ltimes\Wa$, where $\Omega$ is its group of length-zero elements. Conjugation by $\Omega$ permutes $S$, and $\ell(\omega w)=\ell(w)$ for $\omega\in\Omega$ and $w\in\Wa$. The extended Hecke algebra is the corresponding crossed product:
\(
                    \HH=\bigoplus_{\omega\in\Omega}T_\omega\HH_{\mathrm a}.
\)
The bar involution of $\HH$ is determined by $\overline v=v^{-1}$ and $\overline{T_w}=T_{w^{-1}}^{-1}$. For $w\in\Wa$, the Kazhdan--Lusztig basis is normalized by
\[
 \overline{C_w}=C_w,\qquad
 C_w-T_w\in\bigoplus_{\substack{x\in\Wa\\x<w}}v^{-1}\Z[v^{-1}]T_x,
\]
where $<$ denotes the Bruhat order. We extend this basis to $\HH$ by setting $C_{\omega w}=T_\omega C_w$.

We denote the Kazhdan--Lusztig left, right and two-sided preorders on $W$ by $\leq_L$, $\leq_R$ and $\leq_{LR}$, respectively, and the corresponding equivalence relations by $\sim_L$, $\sim_R$ and $\sim_{LR}$. Their equivalence classes are called left, right and two-sided cells \cite{KL79,Lus89}.

A \emph{based ring} is an associative ring $R$ with identity and a specified $\Z$-basis $\{b_i:i\in I\}$ satisfying the following conditions \cite[Section~3.1]{EGNO15}. The structure constants are nonnegative integers, and the identity is a sum of basis elements:
\[
 b_i b_j=\sum_{k\in I}n_{ij}^k b_k,\qquad n_{ij}^k\in\Z_{\geq0},\qquad
 1_R=\sum_{i\in I_0}b_i
\]
for a finite subset $I_0\subset I$. There is an involution $i\mapsto i^*$ of $I$ such that $b_i\mapsto b_{i^*}$ extends to a $\Z$-linear anti-involution of $R$. If $\tau:R\to\Z$ is the linear map defined by $\tau(b_i)=1$ for $i\in I_0$ and $\tau(b_i)=0$ otherwise, then
\[
                         \tau(b_i b_{j^*})=\delta_{ij}.
\]
The index set $I$ need not be finite. An isomorphism of based rings is a ring isomorphism that maps the specified basis bijectively onto the specified basis.

Write
\[
 C_xC_y=\sum_z h_{x,y,z}C_z,\qquad
 \amin(z)=\max_{x,y}\deg_v h_{x,y,z}.
\]
 Let $\gamma_{x,y,z^{-1}}$ be the coefficient of $v^{\amin(z)}$ in $h_{x,y,z}$. The asymptotic ring of $W$ is
\[
 J=\bigoplus_{w\in W}\Z t_w,\qquad
 t_xt_y=\sum_z\gamma_{x,y,z^{-1}}t_z.
\]
With the basis $\{t_w:w\in W\}$, it is a based ring whose anti-involution is $t_w\mapsto t_{w^{-1}}$. For a two-sided cell $\mathbf c$,
\[
 J_{\mathbf c}=\bigoplus_{w\in\mathbf c}\Z t_w
\]
is a two-sided ideal and a direct summand of $J$ as a ring. Its identity is
\begin{equation}\label{eq:cell-unit}
                 1_{\mathbf c}=\sum_{z\in\mathcal D\cap\mathbf c}t_z.
\end{equation}
Here $\mathcal D\subset\Wa$ is the set of distinguished involutions, and each left cell of $W$ contains exactly one element of $\mathcal D$, see \cite{Lus87,Lus89}. Consequently the number of basis elements occurring in \eqref{eq:cell-unit} is the number of left cells in $\mathbf c$.

According to the definition of cells, every left cell of $W$ has the form $\Omega\Gamma$ for a unique left cell $\Gamma$ of $\Wa$, and $(\Omega\Gamma)\cap\Wa=\Gamma$. For every two-sided cell $\mathbf c_{\mathrm a}$ of $\Wa$, the set $\Omega\mathbf c_{\mathrm a}$ is a two-sided cell of $W$, and the numbers of left cells in these two cells agree.

\section{Central extensions and simply connected derived groups}\label{sec:extensions}

All algebraic groups and extensions in this section are defined over $\C$. We first recall how central extensions of a connected reductive group are determined by its derived subgroup. We then give a direct splitting argument and treat the disconnected group that occurs in the symplectic case.

\subsection{Extensions and line bundles}
Let $K$ be an algebraic group. Write $\Extc(K,\Gm)$ for the group of equivalence classes of algebraic central extensions
\begin{equation}\label{eq:central-extension}
 1\longrightarrow\Gm\overset{i}{\longrightarrow}E
   \overset{p}{\longrightarrow}K\longrightarrow1.
\end{equation}
The zero class is the product extension. Thus a class is zero precisely when there is an algebraic group homomorphism $s:K\to E$ with $p\circ s=\id_K$.

Write $\Pic(K)$ for the group of isomorphism classes of line bundles on $K$, with tensor product as the group law and the class of the trivial line bundle as the identity.

Forgetting the multiplication in $E$ gives a $\Gm$-torsor over $K$, and hence a line bundle. This defines a homomorphism
\begin{equation}\label{eq:extension-picard}
             \Extc(K,\Gm)\longrightarrow\Pic(K).
\end{equation}
When $K$ is connected, this map is an isomorphism onto the subgroup of $\Pic(K)$ consisting of primitive line bundle classes, namely those represented by line bundles $\mathcal L$ satisfying
\[
             m^*\mathcal L\simeq
             p_1^*\mathcal L\otimes p_2^*\mathcal L,
\]
where $m:K\times K\to K$ is multiplication and $p_1,p_2$ are the projections; see \cite[Section~1, discussion preceding Theorem~1.1]{Ros21}. Gabber's theorem asserts that every line bundle on a smooth connected unirational algebraic group over $\C$ is primitive \cite[Theorem~1.3]{Ros21}. In particular, this applies to connected reductive groups \cite[Remark~1.4]{Ros21}.

The description of central extensions can be expressed entirely in terms of a semisimple simply connected cover.

\begin{theorem}\label{thm:ext-picard}
Let $K$ be a connected reductive group, let $D=K^{\der}$ be its derived subgroup, and let
\[
                  1\longrightarrow\Pi_D\longrightarrow D^{\mathrm{sc}}
                    \longrightarrow D\longrightarrow1
\]
be the simply connected central cover of $D$. There are isomorphisms
\begin{equation}\label{eq:ext-picard-fundamental}
 \Extc(K,\Gm)\simeq\Pic(K)\simeq\Pic(D)\simeq X^*(\Pi_D),
\end{equation}
where $X^*(\Pi_D)=\operatorname{Hom}_{\mathrm{alg}}(\Pi_D,\Gm)$.
\end{theorem}

\begin{proof}
Since $K$ is connected and reductive, every line bundle on $K$ is primitive by \cite[Theorem~1.3 and Remark~1.4]{Ros21}. Thus the map \eqref{eq:extension-picard} is an isomorphism, giving the first isomorphism in \eqref{eq:ext-picard-fundamental}.

For a connected reductive group, Popov's theorem (\cite[Theorem, p.~794]{Pop23}) states that restriction to the derived subgroup induces an isomorphism of Picard groups. That is, $\iota:D\hookrightarrow K$ induces an isomorphism
\[
                     \iota^*:\Pic(K)\xrightarrow{\sim}\Pic(D).
\]
By Iversen's result (\cite[Proposition~3.8]{Ive76}), central extensions of a semisimple group by a diagonalizable group are classified by homomorphisms from the kernel of its simply connected cover to the diagonalizable group. Taking these groups to be $D$ and $\Gm$, respectively, gives an isomorphism $X^*(\Pi_D)\xrightarrow{\sim}\Extc(D,\Gm)$. Explicitly, writing $q:D^{\mathrm{sc}}\to D$ for the covering map, a character $\chi:\Pi_D\to\Gm$ gives the pushout extension
\[
 E_\chi=(D^{\mathrm{sc}}\times\Gm)\big/
       \{(z,\chi(z)^{-1}):z\in\Pi_D\}\longrightarrow D,
 \qquad [(d,t)]\longmapsto q(d),
\]
whose kernel is the copy of $\Gm$ embedded by $t\mapsto[(1,t)]$. By \cite[Corollary~3.9]{Ive76}, the map $\Extc(D,\Gm)\to\Pic(D)$ that forgets the group structure is also an isomorphism. Hence the composite map, assigning to $\chi$ the line bundle associated with $E_\chi\to D$, gives an isomorphism $X^*(\Pi_D)\xrightarrow{\sim}\Pic(D)$, as required.
\end{proof}

\begin{corollary}\label{cor:ext-splitting}
Let $K$ be a connected reductive group and let $D=K^{\der}$. The following conditions are equivalent:
\begin{enumerate}
\item $D$ is simply connected;
\item $\Pic(K)=0$;
\item every algebraic central extension of $K$ by $\Gm$ splits.
\end{enumerate}
\end{corollary}

\begin{proof}
Let $\Pi_D$ be the kernel of the simply connected cover of $D$. Then $D$ is simply connected if and only if $\Pi_D=1$. Since $\Pi_D$ is finite and diagonalizable, this is equivalent to $X^*(\Pi_D)=0$. By Theorem~\ref{thm:ext-picard}, it is therefore equivalent to $\Pic(K)=0$, proving the equivalence of (1) and (2). The same theorem gives $\Extc(K,\Gm)\simeq\Pic(K)$. Since a central extension splits precisely when its class is zero, (2) and (3) are equivalent.
\end{proof}

\subsection{A direct construction of a splitting}
We give an explicit construction of the splitting in Corollary~\ref{cor:ext-splitting}.

\begin{lemma}\label{lem:basic-splittings}
Every algebraic central extension of a torus by $\Gm$ splits. Every algebraic central extension of a simply connected semisimple group by $\Gm$ has a unique splitting.
\end{lemma}

\begin{proof}
A central extension $E$ of a connected reductive group $K$ by $\Gm$ is connected and reductive. Indeed, connectedness follows from that of the kernel and quotient. The image of the unipotent radical of $E$ is a connected normal unipotent subgroup of $K$, hence is trivial. Thus this radical lies in $\Gm$ and must itself be trivial.

If $K$ is a torus $S$, the derived subgroup of $E$ is contained in the central kernel and must be trivial. Thus $E$ is a torus. The character sequence is
\[
 0\longrightarrow X^*(S)\longrightarrow X^*(E)
   \longrightarrow X^*(\Gm)\longrightarrow0.
\]
It splits because $X^*(\Gm)\simeq\Z$ is free. The anti-equivalence between tori and character lattices gives a splitting of the group extension.

If $K$ is a simply connected semisimple group $D$, the map $E^{\der}\to D$ is a central isogeny. By simple connectedness it is an isomorphism, and its inverse gives a splitting. To see uniqueness, let $s_1,s_2:D\to E$ be two splittings of $p:E\to D$. Since $p\circ s_1=p\circ s_2=\id_D$, the regular map
\[
                     \chi(d)=s_2(d)s_1(d)^{-1}
\]
has image contained in $\ker(p)=\Gm$. The centrality of this kernel gives
\[
 \chi(xy)=s_2(x)\chi(y)s_1(x)^{-1}
          =\chi(y)s_2(x)s_1(x)^{-1}
          =\chi(x)\chi(y).
\]
Thus $\chi:D\to\Gm$ is an algebraic character. Since $\Gm$ is commutative, $\chi$ is trivial on the derived subgroup $[D,D]$. A connected semisimple group satisfies $[D,D]=D$, so $\chi=1$ and hence $s_1=s_2$.
\end{proof}

\begin{lemma}\label{lem:extend-character}
Let $A$ be a finite closed subgroup of a torus $S$. Every algebraic character of $A$ extends to an algebraic character of $S$.
\end{lemma}

\begin{proof}
The groups in $1\to A\to S\to S/A\to1$ are diagonalizable. Applying the character functor gives the exact sequence
\[
 0\longrightarrow X^*(S/A)\longrightarrow X^*(S)
   \longrightarrow X^*(A)\longrightarrow0.
\]
Here the map $X^*(S)\to X^*(A)$ is restriction, $\chi\mapsto\chi|_A$. Its surjectivity means that every character of $A$ extends to a character of $S$.
\end{proof}

\begin{proposition}\label{lem:connected-splitting}
Let $K$ be a connected reductive group whose derived subgroup is simply connected. Every algebraic central extension of $K$ by $\Gm$ splits.
\end{proposition}

\begin{proof}
Given an extension $1\longrightarrow\Gm\overset{i}{\longrightarrow}E
\overset{p}{\longrightarrow}K\longrightarrow1$,  put
\[
                 D=K^{\der},\qquad Z=Z(K)^\circ,
                 \qquad A=Z\cap D.
\]
Then $K=ZD$, and $A$ is finite. By Lemma~\ref{lem:basic-splittings}, the restrictions of the extension to $D$ and $Z$ admit splittings
\[
                         s_D:D\to E,\qquad s_Z:Z\to E.
\]
Their images commute. In fact, for fixed $z\in Z, d\in D$ the commutator $[s_Z(z),s_D(d)]$ lies in the central kernel $\Gm$, and the map
\[
     D\to \Gm,\qquad                    d\longmapsto[s_Z(z),s_D(d)]
\]
is an algebraic character of $D$. Since $D$ is semisimple, this character is trivial.

The two splittings $s_Z, s_D$ may differ on $A=Z\cap D$. Define their difference by
\[
               \phi(a)=s_D(a)s_Z(a)^{-1}\in\Gm
               \qquad(a\in A).
\]
This is a character of $A$. By Lemma~\ref{lem:extend-character}, we can choose a character $\chi:Z\to\Gm$ with $\chi|_A=\phi$, and replace $s_Z$ by
\[
                         s_Z'(z)=i(\chi(z))s_Z(z).
\]
Then $s_Z'$ and $s_D$ agree on $A$, and their images still commute. Consequently the algebraic homomorphism
\[
                 Z\times D\longrightarrow E,
                 \qquad(z,d)\longmapsto s_Z'(z)s_D(d)
\]
is trivial on $\{(a,a^{-1}):a\in A\}$. It descends to a homomorphism $s:K\to E$ through the algebraic quotient
\[
                 (Z\times D)/\{(a,a^{-1}):a\in A\}\simeq K.
\]
For $z\in Z$ and $d\in D$, the definition of $s$ gives
\[
 p(s(zd))=p(s_Z'(z)s_D(d))
         =p(i(\chi(z)))p(s_Z(z))p(s_D(d))=zd,
\]
since $i(\chi(z))\in\ker(p)$ and $s_Z,s_D$ split the restrictions of $p$. Every element of $K$ has the form $zd$, so $p\circ s=\id_K$. Thus $s$ is the desired splitting.
\end{proof}

\subsection{Disconnected groups and projective representations}
The comparison with the Picard group above uses connectedness. We treat the disconnected groups needed here directly.

\begin{lemma}\label{lem:cyclic-splitting}
Let $H$ be a simply connected semisimple group and let $\mu_m$ be a finite cyclic group. Every algebraic central extension of $H\times\mu_m$ by $\Gm$ splits.
\end{lemma}

\begin{proof}
Write the extension as
\[
 1\longrightarrow\Gm\overset{i}{\longrightarrow}E
   \overset{p}{\longrightarrow}H\times\mu_m\longrightarrow1.
\]
By Lemma~\ref{lem:basic-splittings}, the restriction over $H\times\{1\}$ has a unique splitting $s:H\to E$, with $p(s(h))=(h,1)$. Choose a generator $\zeta$ of $\mu_m$ and an element $u\in E$ such that $p(u)=(1,\zeta)$. The algebraic homomorphism $h\mapsto us(h)u^{-1}$ is another splitting over $H\times\{1\}$, since
\[
 p(us(h)u^{-1})=(1,\zeta)(h,1)(1,\zeta)^{-1}=(h,1).
\]
Uniqueness therefore gives $us(h)u^{-1}=s(h)$ for every $h\in H$. Thus $u$ commutes with $s(H)$.

Since $p(u^m)=(1,\zeta^m)=(1,1)$, we can write $u^m=i(c)$ for some $c\in\C^\times$. Choose $a\in\C^\times$ with $a^m=c^{-1}$, and put $v=i(a)u$. The kernel is central, so $v$ still commutes with $s(H)$, and
\[
             v^m=i(a^m)u^m=i(a^m c)=1,
             \qquad p(v)=(1,\zeta).
\]
We can now define
\[
 \sigma:H\times\mu_m\longrightarrow E,
 \qquad \sigma(h,\zeta^j)=s(h)v^j \quad(j\in\Z).
\]
The relation $v^m=1$ makes this well defined, and the commutation of $v$ with $s(H)$ shows that $\sigma$ is an algebraic group homomorphism.  Finally,
\[
 p(\sigma(h,\zeta^j))=p(s(h)v^j)=(h,\zeta^j),
\]
so $p\circ\sigma=\id_{H\times\mu_m}$. Thus $\sigma$ is a splitting of $p$.
\end{proof}

\begin{example}\label{ex:disconnected-extension}
Vanishing of the Picard group does not suffice for a disconnected group. Let
\[
 A=\begin{pmatrix}1&0\\0&-1\end{pmatrix},
 \qquad B=\begin{pmatrix}0&1\\1&0\end{pmatrix}.
\]
Their classes in $\operatorname{PGL}_2$ generate a subgroup $Q\simeq\mu_2\times\mu_2$. The inverse image of $Q$ in $\GL_2$ is a central extension of $Q$ by $\Gm$. Every lift of $[A]$ is a scalar multiple of $A$, and every lift of $[B]$ is a scalar multiple of $B$. Since $AB=-BA$, these lifts cannot commute. If the extension admitted a splitting $s:Q\to\GL_2$, then the commutativity of $[A]$ and $[B]$ in $Q$ would give $s([A])s([B])=s([B])s([A])$, because $s$ is a group homomorphism. But $s([A])$ and $s([B])$ are lifts of $[A]$ and $[B]$, respectively, contradicting the preceding observation. Thus the extension does not split. Nevertheless $\Pic(Q)=0$, since $Q$ is a finite disjoint union of points.
\end{example}

\section{The centralizer of a highest-root triple}\label{sec:centralizer}

We now determine the reductive centralizer for the minimal orbit and apply the splitting results of Section~\ref{sec:extensions}. Choose a Borel subgroup $B\supset T$, with Lie algebra $\bfr$, positive roots $\Phi^+$ and simple roots $\Delta$. Let $\theta$ be the highest root. Choose root vectors $e=e_\theta$ and $f=e_{-\theta}$ so that $(e,h,f)$ is an $\mathfrak{sl}_2$-triple with $h=d\theta^\vee(1)$. By \cite[Theorem~4.3.3]{CM93}, the element $e$ belongs to the minimal nonzero nilpotent orbit. Put
\[
        L=Z_G(\theta^\vee(\Gm)),\qquad H=[L,L].
\]

\begin{proposition}\label{prop:centralizer}
There is a character $\theta_L:L\to\Gm$ whose restriction to $T$ is $\theta$, and
\begin{equation}\label{eq:centralizer-kernel}
                  \Fe=\ker\theta_L,\qquad
                  [\Fe^\circ,\Fe^\circ]=H.
\end{equation}
The group $H$ is simply connected. Moreover,
\[
                            \Fe=\Fe^\circ Z(G).
\]
Except in the symplectic family, $\Fe$ is connected. In that family, including $\SL_2\simeq\Sp_2$,
\begin{equation}\label{eq:symplectic-centralizer}
                        \Fe\simeq\Sp_{2r-2}\times\mu_2
                        \quad(G=\Sp_{2r}).
\end{equation}
\end{proposition}

\begin{proof}
The highest root is long. If a root $\alpha$ satisfies $\langle\alpha,\theta^\vee\rangle=2$, then $\alpha=\theta$. Indeed, the Cauchy--Schwarz inequality and the maximality of the length of $\theta$ give
\[
 \langle\alpha,\theta^\vee\rangle
     =\frac{2(\alpha,\theta)}{(\theta,\theta)}
     \leq 2\frac{\|\alpha\|}{\|\theta\|}\leq2,
\]
and equality forces $\alpha=\theta$. The analogous statement holds for weight $-2$. Thus the weight $2$ and $-2$ spaces for $\theta^\vee$ in $\g$ are $\g_\theta$ and $\g_{-\theta}$.

We use the standard equality $Z_G(h)=Z_G(\theta^\vee(\Gm))=L$, valid over $\C$.

The adjoint action of $L$ commutes with that of $\theta^\vee(\Gm)$ and therefore preserves each of its weight spaces. In particular, it preserves the weight $2$ and $-2$ spaces, which we have shown to be $\g_\theta=\C e$ and $\g_{-\theta}=\C f$. These one-dimensional representations define algebraic characters $\theta_L,\eta:L\to\Gm$ by
\[
 \operatorname{Ad}(l)e=\theta_L(l)e,
 \qquad \operatorname{Ad}(l)f=\eta(l)f\qquad(l\in L).
\]
For $t\in T$, the root-space definition gives $\operatorname{Ad}(t)e=\theta(t)e$, so $\theta_L|_T=\theta$.

To relate the two characters, use $[e,f]=h$ and the fact that $L$ fixes $h$. For every $l\in L$, preservation of the Lie bracket gives
\[
 h=\operatorname{Ad}(l)h
   =[\operatorname{Ad}(l)e,\operatorname{Ad}(l)f]
   =\theta_L(l)\eta(l)h.
\]
Since $h\ne0$, it follows that $\eta(l)=\theta_L(l)^{-1}$. Any element centralizing the triple belongs to $Z_G(h)=L$. Within $L$, an element fixes $e$ exactly when $\theta_L(l)=1$, and then it also fixes $f$. Consequently
\[
             \Fe=Z_G(e,h,f)=\{l\in L:\theta_L(l)=1\}
                 =\ker\theta_L,
\]
which proves the first equality in \eqref{eq:centralizer-kernel}.

The root system of $L$ is
\[
                  \Phi_0=\{\alpha\in\Phi:
                                \langle\alpha,\theta^\vee\rangle=0\}.
\]
Since $\theta$ is dominant, $(\alpha_i,\theta)\geq0$ for every simple root $\alpha_i$. Hence $\langle\alpha_i,\theta^\vee\rangle\geq0$, so $\theta^\vee$ is a dominant cocharacter. Its centralizer $L$ is therefore a connected standard Levi subgroup determined by
\[
                  I=\{\alpha_i\in\Delta:
                               \langle\alpha_i,\theta^\vee\rangle=0\}.
\]
We explain why $H$ is simply connected; see \cite[Remark~5.4.7 and Exercise~6.5.2(iv)]{Con14}. Since $G$ is simply connected, its simple coroots form a $\Z$-basis of $X_*(T)$. Put
\[
                 M_I=\bigoplus_{\alpha_i\in I}\Z\alpha_i^\vee.
\]
Then
\[
 X_*(T)=M_I\oplus
       \bigoplus_{\alpha_i\in\Delta\setminus I}\Z\alpha_i^\vee.
\]
Thus $X_*(T)/M_I$ is free abelian. In particular, $M_I$ is saturated in $X_*(T)$: if a positive integer multiple of an element of $X_*(T)$ lies in $M_I$, then that element already lies in $M_I$.

Let $T_H=(T\cap H)^\circ$, a maximal torus of the semisimple group $H$. The simple coroots of $H$ are the maps $\alpha_i^\vee$ for $\alpha_i\in I$, viewed as cocharacters of $T_H$. Consequently
\[
                   M_I\subseteq X_*(T_H)\subseteq X_*(T),
\]
and both $M_I$ and $X_*(T_H)$ have rank $|I|$. It follows that $X_*(T_H)/M_I$ is finite. But this quotient is a subgroup of the torsion-free group $X_*(T)/M_I$, so it must be zero. Hence
\[
                    X_*(T_H)=M_I.
\]
Finally, a connected semisimple group is simply connected if and only if the cocharacter lattice of a maximal torus is generated by its coroots. Thus $H$ is simply connected.

A character vanishes on the derived subgroup, so $H\subset\ker\theta_L=\Fe$. Since $H$ is connected, it lies in $\Fe^\circ$. Moreover, $H$ is semisimple and $[H,H]=H$, so $H\subset[\Fe^\circ,\Fe^\circ]$.
 Note that $L/H$ is a torus.  The quotient $\Fe/H$ is a subgroup of $L/H$, and is therefore commutative. Hence $[\Fe^\circ,\Fe^\circ]\subset H$. Then we have proved $[\Fe^\circ,\Fe^\circ]= H$.

It remains to determine the components. Let $q:T\to L/H$ be the restriction of the quotient map. Since $L=Z(L)^\circ H$ and $Z(L)^\circ\subset T$, we have $L=TH$, so $q$ is surjective. Its kernel is $T\cap H$. To see that this intersection equals $T_H$, observe that every element of $T\cap H$ centralizes $T_H\subset T$. Since a maximal torus of the connected semisimple group $H$ is its own centralizer, $T\cap H\subset Z_H(T_H)=T_H$. The reverse inclusion follows from the definition of $T_H$. Hence $q$ induces an isomorphism
\[
                         T/T_H\xrightarrow{\sim}L/H.
\]
Pullback along $q$ identifies
\[
 X^*(L/H)=
 \{\lambda\in X^*(T):
       \langle\lambda,\alpha_i^\vee\rangle=0
       \text{ for every }\alpha_i\in I\}.
\]
In one direction, the pullback of a character of $L/H$ is trivial on $T_H$, so it pairs to zero with all the coroots indexed by $I$. Conversely, suppose that $\lambda\in X^*(T)$ satisfies these pairing conditions. Then $\lambda\circ\alpha_i^\vee$ is trivial for every $\alpha_i\in I$. Since these coroots form a basis of $X_*(T_H)$, their images generate $T_H$, and hence $\lambda|_{T_H}=1$. Thus $\lambda$ descends uniquely to a character of $T/T_H$, and the preceding isomorphism makes it a character of $L/H$. This proves the reverse inclusion.

A nonzero element $\lambda$ of a lattice $M$ is called \emph{primitive} if $\lambda=n\mu$, with $\mu\in M$ and a positive integer $n$, implies $n=1$. Equivalently, its coordinates in a $\Z$-basis of $M$ have greatest common divisor $1$.

Since $\theta_L$ is trivial on $H$, it factors through a character $\bar\theta:L/H\to\Gm$, so $\theta_L=\bar\theta\circ\pi$, where $\pi:L\to L/H$ is the quotient map. Its pullback to $T$ is $\theta$. As $\Fe=\ker\theta_L$, we obtain
\[
             \Fe=\pi^{-1}(\ker\bar\theta),
             \qquad \Fe/H=\ker\bar\theta.
\]

We recall the connectedness criterion for the kernel of a torus character. Let $S$ be a torus of dimension $n$ and let $\chi\in X^*(S)$ be nonzero. Write $\chi=d\chi_0$, where $d\geq1$ and $\chi_0$ is primitive. Extend $\chi_0$ to a $\Z$-basis of $X^*(S)$. The corresponding coordinates identify $S$ with $\Gm^n$ and give
\[
 \chi(t_1,\ldots,t_n)=t_1^d,
 \qquad \ker\chi\simeq\mu_d\times\Gm^{n-1}.
\]
Over $\C$, the group $\mu_d$ consists of $d$ distinct points, while $\Gm^{n-1}$ is connected. Thus $\ker\chi$ is connected if and only if $d=1$, which is equivalent to primitivity of $\chi$.

Suppose now that $\theta$ is primitive in $P=X^*(T)$. Then $\bar\theta$ is primitive in $X^*(L/H)$: any expression $\bar\theta=d\psi$ with $d>1$ would, after pullback to $T$, express $\theta$ as $d$ times an element of $P$. The criterion above therefore shows that $\Fe/H=\ker\bar\theta$ is connected. Since both $H$ and $\Fe/H$ are connected, $\Fe$ is connected.

Outside the symplectic family, the highest root has a fundamental-weight coefficient equal to $1$; see Table~\ref{tab:highest-root}, where we use Bourbaki's numbering. Here $A_1=C_1$ and $B_2=C_2$ are included in the symplectic family. In low ranks, we use the conventions that $A_0$ denotes the trivial group, $B_1=A_1$, $D_2=A_1+A_1$, and $D_3=A_3$.

\begin{table}[ht]
\centering
\caption{Highest roots and derived centralizers.}\label{tab:highest-root}
\begin{tabular}{@{}lll@{}}
\toprule
Type of $G$ & $\theta$ & Type of $H$\\
\midrule
$A_r\ (r\geq2)$ & $\omega_1+\omega_r$ & $A_{r-2}$\\
$B_r\ (r\geq3)$ & $\omega_2$ & $A_1+B_{r-2}$\\
$C_r\ (r\geq2)$ & $2\omega_1$ & $C_{r-1}$\\
$D_r\ (r\geq4)$ & $\omega_2$ & $A_1+D_{r-2}$\\
$E_6$ & $\omega_2$ & $A_5$\\
$E_7$ & $\omega_1$ & $D_6$\\
$E_8$ & $\omega_8$ & $E_7$\\
$F_4$ & $\omega_1$ & $C_3$\\
$G_2$ & $\omega_2$ & $A_1$\\
$A_1$ & $2\omega_1$ & $A_0$\\
\bottomrule
\end{tabular}
\end{table}

For every type outside the symplectic family, the table therefore shows that $\theta$ is primitive in $P=X^*(T)$. As proved above, the corresponding character $\bar\theta$ is then primitive in $X^*(L/H)$,  which implies that $\Fe$ is connected. Moreover, $Z(G)\subset\Fe=\Fe^\circ$, and hence $\Fe=\Fe^\circ Z(G)$ in these cases.

It remains to treat the symplectic family. Let $G=\Sp(V)$, $\dim V=2r$, and choose a symplectic basis $v_1,\ldots,v_r,w_1,\ldots,w_r$ adapted to $T$ and $B$. In the standard coordinates for type $C_r$, the highest root is $2\epsilon_1$, where $\epsilon_i$ is the weight of $v_i$. Put
\[
 U=\C v_1\oplus\C w_1,
 \qquad V_0=U^\perp
      =\operatorname{span}\{v_2,\ldots,v_r,w_2,\ldots,w_r\}.
\]
Both subspaces are nondegenerate, so $V=U\oplus V_0$ is an orthogonal direct sum. The root spaces $\g_{\pm\theta}$ act only on $U$. After rescaling $v_1$ and $w_1$ by reciprocal factors, the chosen triple acts on $U$, in the ordered basis $(v_1,w_1)$, as
\[
 e|_U=\begin{pmatrix}0&1\\0&0\end{pmatrix},
 \qquad h|_U=\begin{pmatrix}1&0\\0&-1\end{pmatrix},
 \qquad f|_U=\begin{pmatrix}0&0\\1&0\end{pmatrix},
\]
and all three endomorphisms vanish on $V_0$.

Since $\Sp(U)=\SL(U)\simeq\SL_2$, we have an algebraic embedding
\[
 \rho:\SL_2\longrightarrow\Sp(V),
 \qquad g\longmapsto g\oplus\id_{V_0},
\]
whose differential sends the standard $\mathfrak{sl}_2$-triple to $(e,h,f)$. We use this embedding for an $\SL_2$-action on $V$. Then $U$ is the standard two-dimensional module, and $V_0$ is a direct sum of $2r-2$ copies of the trivial $\SL_2$-module, so $V_0$ is precisely the subspace fixed by $\rho(\SL_2)$. An element of $G$ centralizes $(e,h,f)$ if and only if it commutes with $\rho(\SL_2)$. Indeed, $\rho(\SL_2)$ is generated by the one-parameter subgroups $\exp(te)$ and $\exp(tf)$, and the reverse implication follows by differentiation.

Let $g\in\Fe$. Then $g$ preserves $V_0$ since $V_0$ is the subspace fixed by $\rho(\SL_2)$ and $g$ commutes with $\rho(\SL_2)$. Moreover, $g\in\Fe\subset G=\Sp(V)$, so $g$ preserves the symplectic form on $V$, which we denote by $\langle\ ,\ \rangle$. For $u\in U=V_0^\perp$ and $v\in V_0$, we have $g^{-1}v\in V_0$ and hence
\[
             \langle gu,v\rangle
                 =\langle u,g^{-1}v\rangle=0.
\]
Thus $gU\subseteq V_0^\perp=U$, and invertibility of $g$ gives $gU=U$. By Schur's lemma, its restriction to the irreducible module $U$ is a scalar $a\id_U$. Since the form on $U$ is nondegenerate and $g$ preserves it, we have
\[
       \langle v_1,w_1\rangle
          =\langle av_1,aw_1\rangle
          =a^2\langle v_1,w_1\rangle,
       \qquad a^2=1.
\]
Its restriction to $V_0$ lies in $\Sp(V_0)$, and every element of $\Sp(V_0)$ commutes with the trivial $\SL_2$-action on $V_0$. Conversely, every block diagonal map of this form is symplectic and centralizes the triple. We therefore obtain an isomorphism
\[
 \Sp(V_0)\times\mu_2\xrightarrow{\sim}\Fe,
 \qquad (A,\varepsilon)\longmapsto
          \varepsilon\id_U\oplus A.
\]
Since $\dim V_0=2r-2$, this proves \eqref{eq:symplectic-centralizer}.

The group $\Sp(V_0)$ is connected, so $\Fe^\circ$ corresponds to $\Sp(V_0)\times\{1\}$, and $\Fe$ has two connected components distinguished by $\varepsilon$. The central element $-\id_V\in Z(G)$ corresponds to $(-\id_{V_0},-1)$ and therefore lies in the other component. Multiplication by this element carries $\Fe^\circ$ onto that component, proving $\Fe=\Fe^\circ Z(G)$. For $r=1$, we have $V_0=0$ and $\Sp(V_0)=1$, so the formula gives $\Fe=\mu_2$.
\end{proof}

\begin{remark}\label{rem:type-A}
In type $A_r$, $r\geq2$, the decomposition $\C^{r+1}=\C^2\oplus\C^{r-1}$ gives
\begin{equation}\label{eq:type-A}
 \Fe\simeq\{(a,A)\in\Gm\times\GL_{r-1}:a^2\det A=1\}.
\end{equation}
Its derived subgroup is $\SL_{r-1}$. In particular, $\Fe$ need not be isomorphic to $\GL_{r-1}$. For instance, for $G=\SL_4$ one has $\Fe\simeq\Gm\times\SL_2$, via $(a,S)\mapsto(a,a^{-1}S)$, whereas the center of $\GL_2$ is connected and that of $\Gm\times\SL_2$ is not. Thus the representation ring in Theorem~\ref{thm:matrix} depends on the group itself, not only on its root system.
\end{remark}

\begin{corollary}\label{cor:centralizer-splitting}
For the minimal nilpotent element in $G$, one has
\[
                   \Pic(\Fe^\circ)=\Pic(\Fe)=0,
                   \qquad \Extc(\Fe,\Gm)=0.
\]
\end{corollary}

\begin{proof}
Proposition~\ref{prop:centralizer} shows that the connected reductive group $\Fe^\circ$ has simply connected derived subgroup. Hence Corollary~\ref{cor:ext-splitting} gives $\Pic(\Fe^\circ)=0$. Every connected component of $\Fe$ is isomorphic to $\Fe^\circ$ by translation, so its Picard group vanishes as well. Since there are finitely many components, $\Pic(\Fe)=0$.

If $\Fe$ is connected, Proposition~\ref{lem:connected-splitting} gives $\Extc(\Fe,\Gm)=0$. Otherwise \eqref{eq:symplectic-centralizer} and Lemma~\ref{lem:cyclic-splitting} give the same conclusion, including $\Fe=\mu_2$ in rank one. 
\end{proof}

Proposition~\ref{prop:centralizer} and Corollary~\ref{cor:centralizer-splitting} prove Theorem~\ref{thm:centralizer-structure}. In particular, the algebraic projective representations of $\Fe$ can be lifted to algebraic representations. We next explain their role in the asymptotic ring.

\section{The finite set and the asymptotic ring}\label{sec:geometry}

For an algebraic group $F$ acting on a finite set $Y$, let $K_F(Y\times Y)$ be the Grothendieck group of $F$-equivariant vector bundles. It has the convolution product
\begin{equation}\label{eq:convolution}
               (E*E')_{y,z}
                   =\bigoplus_{x\in Y}E_{y,x}\otimes E'_{x,z}.
\end{equation}
Its natural basis consists of irreducible equivariant bundles.

Bezrukavnikov and Ostrik describe the asymptotic ring $J_{\mathbf c_e}$ by convolution of equivariant vector bundles on the square of a finite set, allowing central extensions in the equivariance data; see \cite[Sections~4.2, 5.1 and Theorem~4]{BO04}. We recall this construction. For each $y\in Y$, write $F_y=\{g\in F:gy=y\}$ for its stabilizer and equip it with an algebraic central extension
\[
            1\longrightarrow\Gm\overset{i_y}{\longrightarrow}\widetilde F_y
               \longrightarrow F_y\longrightarrow1.
\]
These extensions are required to be compatible with conjugation by $F$.

For an ordinary $F$-equivariant vector bundle on $Y$, the fiber at $y$ is a linear representation of $F_y$. In the extended setting, the fiber $V_y$ is instead an algebraic representation $\widetilde\rho_y:\widetilde F_y\to\GL(V_y)$ satisfying
\[
          \widetilde\rho_y(i_y(\lambda))=\lambda\id_{V_y}
          \qquad(\lambda\in\Gm).
\]
For $V_y\ne0$, passage to projective linear transformations gives a well-defined homomorphism
\[
                    F_y\longrightarrow\operatorname{PGL}(V_y).
\]
Two lifts of the same element of $F_y$ differ by an element of the central kernel and hence act by linear maps differing only by a nonzero scalar. This is the projective action of the stabilizer. If the extension splits, restricting $\widetilde\rho_y$ to a splitting gives an ordinary algebraic representation of $F_y$.

For the construction of the corresponding convolution category on $Y\times Y$, we refer to \cite[Section~5.1]{BO04}. We denote its Grothendieck ring, with the basis of irreducible objects, by $K_F^{\mathrm{tw}}(Y\times Y)$. We will only need the case in which all point extensions split. In that case, a choice of splittings compatible with the $F$-action gives a based-ring isomorphism
\[
                  K_F^{\mathrm{tw}}(Y\times Y)
                         \simeq K_F(Y\times Y).
\]

We use the following description of the asymptotic ring. In addition to the based-ring isomorphism, it records the action of $Z(G)$ and the cardinality of the finite set $Y_e$.

\begin{theorem}\label{thm:geometric-input}
Let $G$ be as above, let $e\in\g$ be nilpotent, let $\Fe=Z_G(e,h,f)$, and let $\mathbf c_e$ be the associated two-sided cell in $W$. There is a finite $\Fe$-set $Y_e$, with central extensions of its stabilizers, such that
\begin{enumerate}
\item $J_{\mathbf c_e}\simeq K_{\Fe}^{\mathrm{tw}}(Y_e\times Y_e)$ as based rings;
\item $Z(G)$ acts trivially on $Y_e$;
\item $|Y_e|=\rank_\Z K_0(\Be)$.
\end{enumerate}
Here $K_0$ denotes the Grothendieck group of coherent sheaves.
\end{theorem}

The based-ring statement is the theorem of Bezrukavnikov--Ostrik \cite[Theorem~4]{BO04}. We now explain (2) and (3).

Let $\mathcal N\subset\g$ be the nilpotent cone and let $A$ be the $\OO(\mathcal N)$-algebra of the noncommutative Springer resolution \cite[Section~1.5.3]{BM13}. If $\mathfrak m_e\subset\OO(\mathcal N)$ is the maximal ideal of $e$, put $A_e=A/\mathfrak m_e A$, and write $A_e\text{-}\mathrm{mod}$ for the category of finite-dimensional left $A_e$-modules. By the geometric realization described in \cite[Section~2.5]{BDD23}, the finite set in (1) can be chosen as
\[
                            Y_e=\Irr(A_e).
\]
Let $M$ be a simple $A_e$-module, and write $y=[M]\in Y_e$ for its isomorphism class. An element $g\in F_y$ preserves this class, so there is an invertible linear map $T_g:M\to M$ satisfying $T_g(av)=g(a)T_g(v)$ for $a\in A_e$ and $v\in M$. By Schur's lemma, $T_g$ is unique up to a nonzero scalar. Thus $[T_g]\in\operatorname{PGL}(M)$ defines the projective stabilizer action. The central elements of $G$ fix its isomorphism class; this is the center-triviality used in \cite[proof of Lemma~1]{BDD23}.

For (3), the noncommutative Springer resolution gives
\begin{equation}\label{eq:K-comparison}
                  K_0(\Be)\simeq K_0(A_e\text{-}\mathrm{mod});
\end{equation}
see \cite[Section~5.1.8, equation~(9)]{BM13}. Since $A_e$ is finite-dimensional, the classes of its simple modules form a $\Z$-basis of $K_0(A_e\text{-}\mathrm{mod})$. Together with the identification $Y_e=\Irr(A_e)$ from \cite[Section~2.5]{BDD23}, this yields
\[
                  |Y_e|=|\Irr(A_e)|=\rank_\Z K_0(\Be),
\]
which proves (3).

\begin{corollary}\label{cor:trivial-action}
If $\Fe=\Fe^\circ Z(G)$, then the $\Fe$-action on $Y_e$ is trivial.
\end{corollary}

\begin{proof}
A connected algebraic group acts trivially on a finite set. Theorem~\ref{thm:geometric-input}(2) treats the remaining central elements.
\end{proof}

The preceding discussion immediately gives the following consequence.

\begin{corollary}\label{cor:matrix-criterion}
Let $e\in\g$ be nilpotent. Suppose that $\Fe=\Fe^\circ Z(G)$ and $\Extc(\Fe,\Gm)=0$. Then
\[
             J_{\mathbf c_e}\simeq\Mat_M(R(\Fe)),
             \qquad M=\rank_\Z K_0(\Be),
\]
as based rings.
\end{corollary}

\section{The size of the finite set}\label{sec:springer}

We return to the minimal nilpotent element $e=e_\theta$. Under the identification $G/B\to\{\text{Borel subalgebras of }\g\}$, $gB\mapsto\Ad(g)\bfr$, we regard $\Be$ as a closed subvariety of $G/B$. Let $X(w)=\overline{BwB/B}$ be the Schubert variety indexed by $w\in\Wf$, and set
\[
                   S_\theta=\{w\in\Wf:w^{-1}\theta\in\Phi^+\}.
\]

Abe--Crooks \cite[Sections~3.1--3.2]{AC16} describe $\Be$ as a union of Schubert cells and compute its Euler characteristic. We record these results together with the rank of $K_0(\Be)$.

\begin{proposition}\label{prop:springer}
The Springer fiber $\Be$ is the union of the Schubert cells $BwB/B$ for $w\in S_\theta$, and
\begin{equation}\label{eq:half}
                 \rank_\Z K_0(\Be)=\chi(\Be)
                         =|S_\theta|=\frac{|\Wf|}{2}.
\end{equation}
\end{proposition}

\begin{proof}
For the decomposition into Schubert cells, we follow the argument of Abe--Crooks \cite[Section~3.1]{AC16}. The group $B$ preserves the highest-root line $\C e_\theta$, so $\Be$ is $B$-stable. The point $wB$ belongs to $\Be$ precisely when $\Ad(w^{-1})e_\theta\in\bfr$, equivalently $w^{-1}\theta\in\Phi^+$. By $B$-stability, the entire cell $BwB/B$ lies in $\Be$ if and only if $wB\in\Be$. Thus
\[
                  \Be=\bigsqcup_{w\in S_\theta}BwB/B.
\]
These Schubert cells give an affine paving, and hence $\chi(\Be)=|S_\theta|$.

We next deduce the rank equality. Since $\Be$ is closed, it contains $X(w)=\overline{BwB/B}$ for every $w\in S_\theta$. Localization for coherent $K$-groups along the affine paving shows that the classes $[\OO_{X(w)}]$, $w\in S_\theta$, generate $K_0(\Be)$. By \cite[Theorem~3.4.1(ii)]{Bri05}, the Schubert structure-sheaf classes form a $\Z$-basis of $K_0(G/B)$. The pushforwards of $[\OO_{X(w)}]$, $w\in S_\theta$, are distinct members of this basis, so these classes are linearly independent in $K_0(\Be)$. Thus they form a $\Z$-basis of $K_0(\Be)$, proving the rank equality.

The formula $\chi(\Be)=|\Wf|/2$ is proved in \cite[Section~3.2]{AC16} by counting long roots. It can also be seen from the fixed-point-free involution $w\mapsto s_\theta w$ of $\Wf$, since
\[
                       (s_\theta w)^{-1}\theta=-w^{-1}\theta.
\]
It interchanges $S_\theta$ and its complement, proving the last equality in \eqref{eq:half}.
\end{proof}

\begin{corollary}\label{cor:Y-size}
For the finite set in Theorem~\ref{thm:geometric-input} attached to the minimal orbit,
\[
                              |Y_e|=\frac{|\Wf|}{2}.
\]
\end{corollary}

\begin{proof}
Combine Theorem~\ref{thm:geometric-input}(3) with Proposition~\ref{prop:springer}.
\end{proof}

\begin{remark}\label{rem:bruhat-ideal}
The set $S_\theta$ is a lower order ideal in the Bruhat order: if $v\leq w$ and $w\in S_\theta$, then $v\in S_\theta$. This also admits a direct root-theoretic proof. It suffices to consider a cover $v<w$, with $\ell(w)=\ell(v)+1$. Write $w=s_\beta v$ for $\beta\in\Phi^+$. The length inequality implies $v^{-1}\beta\in\Phi^+$. Since the highest root $\theta$ is dominant, we have $m:=\langle\theta,\beta^\vee\rangle\geq0$. Using $w^{-1}=v^{-1}s_\beta$ and the reflection formula, we obtain
\[
\begin{aligned}
 w^{-1}\theta
   &=v^{-1}s_\beta(\theta)\\
   &=v^{-1}\bigl(\theta-\langle\theta,\beta^\vee\rangle\beta\bigr)\\
   &=v^{-1}\theta-m\,v^{-1}\beta.
\end{aligned}
\]
That is,
\(
                   v^{-1}\theta=w^{-1}\theta+m\,v^{-1}\beta.
\)
If $w\in S_\theta$, then $v^{-1}\theta$ is a positve root, proving the assertion.
\end{remark}

\section{Proof of the main theorems}\label{sec:proof}

\begin{proof}[Proof of Theorem~\ref{thm:matrix}]
Theorem~\ref{thm:centralizer-structure} gives $\Fe=\Fe^\circ Z(G)$ and $\Extc(\Fe,\Gm)=0$. Corollary~\ref{cor:matrix-criterion} therefore gives an isomorphism of based rings
\[
                  J_{\cmin}\simeq\Mat_N(R(\Fe)),
                  \qquad N=\rank_\Z K_0(\Be).
\]
Proposition~\ref{prop:springer} identifies this rank with $\chi(\Be)=|\Wf|/2$.
\end{proof}

\begin{proof}[Proof of Theorem~\ref{thm:count}]
The identity of the matrix based ring in Theorem~\ref{thm:matrix} is
\[
                          \sum_{i=1}^{N}E_{ii}\otimes[\C],
                          \qquad N=|\Wf|/2.
\]
Its $N$ summands are distinct basis elements. A based-ring isomorphism preserves the basis expansion of the identity. Comparing with \eqref{eq:cell-unit} gives
\[
                        |\mathcal D\cap\cmin|=N.
\]
Since each left cell contains exactly one distinguished involution, $\cmin$ contains $N$ left cells. This proves the assertion for $W$.
\end{proof}

\section*{Disclosure of AI assistance}

The first author had developed the overall outline of the argument before using ChatGPT~6 (OpenAI) in the preparation of this paper. He had also presented this work on several occasions, including a recent talk at Hebei Normal University on July~19, 2026. 

ChatGPT assisted us in developing the constructive splitting proofs in Section~\ref{sec:extensions} and the uniform proof of connectedness of the reductive centralizer outside the symplectic family in Section~\ref{sec:centralizer}, replacing our earlier type-by-type arguments for connectedness. It also assisted with revising the exposition. The first author has personally checked all mathematical arguments and the final text. The authors take full responsibility for the contents of this paper.

\bibliographystyle{amsalpha}
\bibliography{references}

\providecommand{\bysame}{\leavevmode\hbox to3em{\hrulefill}\thinspace}
\providecommand{\MR}{\relax\ifhmode\unskip\space\fi MR }
\providecommand{\MRhref}[2]{%
  \href{http://www.ams.org/mathscinet-getitem?mr=#1}{#2}
}
\providecommand{\href}[2]{#2}
\begin{thebibliography}{EGNO15}

\bibitem[AC16]{AC16}
Hiraku Abe and Peter Crooks, \emph{Hessenberg varieties for the minimal
  nilpotent orbit}, Pure Appl. Math. Q. \textbf{12} (2016), no.~2, 183--223.

\bibitem[BDD23]{BDD23}
Roman Bezrukavnikov, Stefan Dawydiak, and Galyna Dobrovolska, \emph{On the
  structure of the affine asymptotic {Hecke} algebras}, Transform. Groups
  \textbf{28} (2023), no.~3, 1059--1079, With an appendix by Roman
  Bezrukavnikov, Alexander Braverman, and David Kazhdan.

\bibitem[Bez09]{Bez09}
Roman Bezrukavnikov, \emph{Perverse sheaves on affine flags and nilpotent cone
  of the {Langlands} dual group}, Israel J. Math. \textbf{170} (2009),
  185--206.

\bibitem[BM13]{BM13}
Roman Bezrukavnikov and Ivan Mirkovi{\'c}, \emph{Representations of semisimple
  {Lie} algebras in prime characteristic and the noncommutative {Springer}
  resolution}, Ann. of Math. (2) \textbf{178} (2013), no.~3, 835--919, With an
  appendix by Eric Sommers.

\bibitem[BO04]{BO04}
Roman Bezrukavnikov and Viktor Ostrik, \emph{On tensor categories attached to
  cells in affine {Weyl} groups. {II}}, Representation theory of algebraic
  groups and quantum groups, Adv. Stud. Pure Math., vol.~40, Math. Soc. Japan,
  Tokyo, 2004, pp.~101--119.

\bibitem[Bri05]{Bri05}
Michel Brion, \emph{Lectures on the geometry of flag varieties}, Topics in
  cohomological studies of algebraic varieties (Piotr Pragacz, ed.), Trends in
  Mathematics, Birkh{\"a}user, Basel, 2005, pp.~33--85.

\bibitem[CM93]{CM93}
David~H. Collingwood and William~M. McGovern, \emph{Nilpotent orbits in
  semisimple {Lie} algebras}, Van Nostrand Reinhold Mathematics Series, Van
  Nostrand Reinhold, New York, 1993.

\bibitem[Con14]{Con14}
Brian Conrad, \emph{Reductive group schemes}, Autour des sch{\'e}mas en
  groupes. Vol. I, Panor. Synth{\`e}ses, vol. 42/43, Soc. Math. France, Paris,
  2014, pp.~93--444.

\bibitem[EGNO15]{EGNO15}
Pavel Etingof, Shlomo Gelaki, Dmitri Nikshych, and Victor Ostrik, \emph{Tensor
  categories}, Mathematical Surveys and Monographs, vol. 205, American
  Mathematical Society, Providence, RI, 2015.

\bibitem[Ive76]{Ive76}
Birger Iversen, \emph{The geometry of algebraic groups}, Adv. Math. \textbf{20}
  (1976), no.~1, 57--85.

\bibitem[KL79]{KL79}
David Kazhdan and George Lusztig, \emph{Representations of {Coxeter} groups and
  {Hecke} algebras}, Invent. Math. \textbf{53} (1979), no.~2, 165--184.

\bibitem[Lus87]{Lus87}
George Lusztig, \emph{Cells in affine {Weyl} groups. {II}}, J. Algebra
  \textbf{109} (1987), no.~2, 536--548.

\bibitem[Lus89]{Lus89}
\bysame, \emph{Cells in affine {Weyl} groups. {IV}}, J. Fac. Sci. Univ. Tokyo
  Sect. IA Math. \textbf{36} (1989), no.~2, 297--328.

\bibitem[Par26]{Par26}
James Parkinson, \emph{Lusztig's asymptotic algebra and a cell factorisation
  for the second lowest two-sided cell in type {$\widetilde B_n$}}, Preprint,
  \href{https://arxiv.org/abs/2610.25771}{arXiv:2610.25771}, 2026.

\bibitem[Pop23]{Pop23}
Vladimir~L. Popov, \emph{Picard group of a connected affine algebraic group},
  Russian Math. Surveys \textbf{78} (2023), no.~4, 794--796.

\bibitem[Qiu25]{Qiu25}
Yannan Qiu, \emph{The based rings of two-sided cells in an affine {Weyl} group
  of type {$\widetilde B_3$}, {IV}}, Preprint,
  \href{https://arxiv.org/abs/2505.02459}{arXiv:2505.02459}, 2025.

\bibitem[QX23]{QX23}
Yannan Qiu and Nanhua Xi, \emph{The based rings of two-sided cells in an affine
  {Weyl} group of type {$\widetilde B_3$}, {I}}, Sci. China Math. \textbf{66}
  (2023), no.~2, 221--236.

\bibitem[Ros21]{Ros21}
Zev Rosengarten, \emph{Translation-invariant line bundles on linear algebraic
  groups}, J. Algebraic Geom. \textbf{30} (2021), no.~3, 433--455.

\bibitem[Shi87]{Shi87}
Jian-yi Shi, \emph{A two-sided cell in an affine {Weyl} group}, J. London Math.
  Soc. (2) \textbf{36} (1987), no.~3, 407--420.

\bibitem[Shi11]{Shi11}
\bysame, \emph{The second lowest two-sided cell in an affine {Weyl} group}, J.
  Algebra \textbf{333} (2011), 161--179.

\bibitem[Shi23]{Shi23}
\bysame, \emph{The second lowest two-sided cell in the affine {Weyl} group
  {$\widetilde B_n$}}, J. Combin. Theory Ser. A \textbf{196} (2023), 105736, 32
  pp.

\bibitem[Xi90]{Xi90}
Nanhua Xi, \emph{The based ring of the lowest two-sided cell of an affine
  {Weyl} group}, J. Algebra \textbf{134} (1990), no.~2, 356--368.

\bibitem[Xi02]{Xi02}
\bysame, \emph{The based ring of two-sided cells of affine {Weyl} groups of
  type {$\widetilde A_{n-1}$}}, vol. 157, Mem. Amer. Math. Soc., no. 749,
  American Mathematical Society, Providence, RI, 2002.

\end{thebibliography}
\end{document}